\documentclass{amsart}
\usepackage[utf8]{inputenc}
\usepackage{enumitem}
\usepackage{amsfonts}
\usepackage{amsmath,amsthm}
\usepackage{comment}
\usepackage{listings}
\usepackage{graphics}
\usepackage{bm}
\usepackage[inline]{asymptote}
\usepackage{graphicx,xcolor}
\graphicspath{{./figures/}}
\usepackage{hyperref,url}
\definecolor{darkblue}{rgb}{0,0,0.75}
\definecolor{darkred}{rgb}{0.75,0,0}
\definecolor{darkgreen}{rgb}{0,0.75,0}   
\hypersetup{colorlinks,
  filecolor=black,
  linkcolor=darkblue,
  citecolor=darkgreen,
  urlcolor=darkred,
  bookmarksopen=true}
\usepackage{amsthm}
\usepackage{graphicx}
\usepackage{amssymb}
\usepackage[section]{placeins}
\usepackage{cite}
\usepackage{appendix}
\usepackage{indentfirst}
\usepackage{extarrows}

\numberwithin{equation}{section}
\newtheorem*{question}{Question}

\newtheorem{thm}{Theorem}[section]
\newtheorem{lem}[thm]{Lemma}
\newtheorem{prop}[thm]{Proposition}

\theoremstyle{definition}
\newtheorem{defi}[thm]{Definition}

\swapnumbers

\numberwithin{equation}{section}

\newcommand{\att}{(\cdot,t)}

\title{Huisken's Distance Comparison Principle in Higher Codimension}
\author{Qi Sun}
\thanks{This work was partially supported by NSF grant DMS-2348305.}
\date{\today}
\address{Qi Sun, Department of Mathematics, University of Wisconsin-Madison}
\email{qsun79@wisc.edu}
\begin{document}
\begin{abstract}
We establish a variant of Huisken's distance comparison principle for reflection symmetric immersed Curve Shortening flow in $\mathbb R^n,n\geq2$.

As an application, we show that certain symmetric Curve Shortening flow with a one-to-one convex projection develops a Type~I singularity and becomes asymptotically circular. 
\end{abstract}
\maketitle
\section{Introduction}
\subsection{Setup}
\subsubsection*{Huisken's distance ratio}
Following Huisken\cite{huisken1998distance}, for a closed space curve $\gamma:S^1\rightarrow\mathbb R^n$, let us define the \emph{extrinsic distance} to be
\begin{equation}
    d=d(p,q)=|\gamma(p)-\gamma(q)|,
\end{equation}
the \emph{intrinsic distance} $l=l(p,q)$ to be the length of one arc between two different points $\gamma(p),\gamma(q)$ and a \emph{modification of the intrinsic distance} to be
\begin{equation}
\label{definition of psi}
\psi=\psi(p,q)=\frac{L}{\pi}\sin(\frac{\pi l}{L}),  
\end{equation}
where $L$ is the (total) length of the curve $\gamma$. Here $\psi$ is well defined and is independent of the choice of the two arcs between $\gamma(p)$ and $\gamma(q)$ because of the identity $\sin\alpha=\sin(\pi-\alpha)$.

We define \emph{Huisken's distance ratio} between the points $\gamma(p)$ and $\gamma(q)$ to be 
\begin{equation}
\label{definition of Huisken's distance ratio}
    \frac{d}{\psi}(p,q)=\frac{|\gamma(p)-\gamma(q)|}{\frac{L}{\pi}\sin(\frac{\pi l}{L})}
\end{equation}
and we denote by 
\begin{equation}
    \min\frac{d}{\psi}:=\inf_{\substack{(p,q)\in S^1\times S^1\\p \neq q}} \frac{d}{\psi}(p,q)
\end{equation}
the \emph{minimum value} of Huisken's distance ratio.

\subsubsection*{\texorpdfstring{Curve Shortening flow in $\mathbb R^n$}{Curve Shortening flow in higher codimension}}
Let us consider the Space\footnote{The word “space” is used only to emphasize that we are working in higher codimension; it does not denote a different flow.} Curve Shortening flow (SCSF),
\begin{equation}
\label{equation of CSF}
    \gamma_t=\gamma_{ss}
\end{equation}
where $\gamma:S^1\times \left[0,T\right)\rightarrow\mathbb{R}^n$ is smooth ($S^1=\mathbb R/2\pi\mathbb Z$), $u\rightarrow \gamma(u,t)$ is an immersion and $\partial_s=\frac{\partial}{\partial s}$ is the derivative with respect to arc-length, defined by
\begin{equation}
\label{the equation defining the arc length}
    \frac{\partial}{\partial s}:=\frac{1}{|\gamma_u|}\frac{\partial}{\partial u}.
\end{equation}

Along SCSF $\gamma\att$, Huisken's distance ratio can be considered at each time $t$ and we use $\min\frac{d}{\psi}$ to denote the minimum value of Huisken's distance ratio at time $t$, with the dependence on time $t$ left implicit.
\subsection{Background}
Huisken \cite{huisken1998distance} proved that $\min\frac{d}{\psi}$ is non-decreasing for planar embedded CSF and gave an alternative proof of the Gage-Hamilton-Grayson theorem \cite{GageHamilton,Grayson} based on the blow-up results. 
In \cite{AndrewsBryan+2011+179+187} Andrews and Bryan refined the argument of Huisken and gave a direct proof of the Gage-Hamilton-Grayson theorem. 
In \cite{bryan2023sharp} Bryan, Langford and Zhu proved that methods of \cite{AndrewsBryan+2011+179+187} work for CSF on the two dimensional sphere. See  \cite{edelen2015noncollapsing,johnson24singularity} who extended Huisken's distance comparison principle to CSF in surfaces. See also \cite{langford2023distance,bian2024distance}.

In the higher codimensional case, one \emph{cannot} expect $\min\frac{d}{\psi}$ to be non-decreasing for all SCSF starting with embedded curves. Altschuler \cite{altschuler1991singularities} pointed out that there are embedded space curves that evolve to have self-intersections under SCSF at some positive time. At such time, $\min\frac{d}{\psi}$ vanishes because of self-intersections. So $\min\frac{d}{\psi}$ is positive at initial time but becomes zero later, thus is not non-decreasing. Even for SCSF that remains embedded, one cannot expect such monotonicity because the multiplicity of the blow-up limit could be two or higher. Examples in \cite[\S3]{altschuler2013zoo} can be used to illustrate this explicitly.

Though the distance comparison principle cannot be generalized to SCSF directly, the assumption of reflection symmetry allows us to establish a variant of this principle that we explain now.
\subsection{Main result}
Let $P_0\subset\mathbb R^n$ be a hyperplane with $R_0:\mathbb R^n\rightarrow\mathbb R^n$ the reflection about $P_0$. 
\begin{defi}
    We say that a curve $\gamma:S^1\rightarrow\mathbb R^n$ is \emph{symmetric two-crossing} if it is $R_0$ invariant and it intersects the hyperplane $P_0$ exactly twice. 
\end{defi}

\begin{defi}
\label{loose definition of the symmetric distance ratio}
For a symmetric two-crossing curve $\gamma:S^1\rightarrow\mathbb R^n$, we define the \emph{(reflection) symmetric distance ratio} $I$ at $\gamma(p)$ to be
\begin{equation}
    I(p)=\frac{d}{\psi}(p,R_0(p))
\end{equation}
where $R_0(p)$ denotes the parameter such that\footnote{For curves with self-intersections, the condition $\gamma(R_0(p))=R_0\gamma(p)$ does not uniquely define $R_0(p)$, because there are points that are reflection symmetric extrinsically but not intrinsically. An unambigous definition of $R_0(p)$ is given in equation  (\ref{notion of symmetric points}).} $\gamma(R_0(p))=R_0\gamma(p)$ and $\frac{d}{\psi}$ is defined in equation (\ref{definition of Huisken's distance ratio}).
\end{defi}
We denote by $\gamma:S^1\times \left[0,T\right)\rightarrow\mathbb{R}^n$ the solution to the SCSF with the initial condition $\gamma(u,0)=\gamma_0(u)$, by $I\att$ the symmetric distance ratio of SCSF $\gamma\att$ and by $(\min I)(t)$ the minimum value of the ratio $I(\cdot,t)$.  Our main result is:
\begin{thm}
\label{n dim-monotonicity of Huisken's distance ratio restricted to symmetric points}
If the initial curve $\gamma_0$ is symmetric two-crossing, then the SCSF $\gamma\att$ is also symmetric two-crossing for each time $t\in[0,T)$. More importantly, the minimum of the symmetric distance ratio $(\min I)(t)$ is positive and non-decreasing with respect to time $t$.
\end{thm}
We do not require a symmetric two-crossing curve $\gamma$ to be embedded. Even if the initial curve $\gamma_0$ is $R_0$ invariant and embedded, SCSF $\gamma\att$ still can develop self-intersections, which arise in pairs and do not lie on the hyperplane $P_0$.

Let us explain why we require that the initial curve $\gamma_0$ intersects $P_0$ exactly twice.
If the initial curve $\gamma_0$ is $R_0$ invariant but it intersects the hyperplane $P_0$ strictly more than twice, then the SCSF $\gamma\att$ has a self-intersection on the hyperplane $P_0$ and $(\min I)(t)$ is zero which is monotonic trivially, unless the intersection number between $\gamma\att$ and $P_0$ decreases to two. The intersection number is non-increasing by the Sturmian theorem \cite{angenent1988zero} (see also \cite[Lemma 5.2]{sun2024curve}).

We will prove Theorem \ref{n dim-monotonicity of Huisken's distance ratio restricted to symmetric points} in \S \ref{the section on adaptions of Huisken's distance comparison principle}. The point we wish to make here is that Theorem \ref{n dim-monotonicity of Huisken's distance ratio restricted to symmetric points} can be proved using the arguments as those presented in \cite{huisken1998distance}, with adjustments specific to the symmetric case.
\subsection{Application}

\subsubsection*{Convex projection}
Let $P_{xy}:\mathbb{R}^n=\mathbb{R}^2\times\mathbb R^{n-2}\rightarrow\mathbb{R}^2$ be the orthogonal projection onto the first two coordinates, which we call $x$ and $y$. For a space curve $\gamma$, let $P_{xy}|_\gamma:\gamma\rightarrow xy$-plane be its restriction to $\gamma$.

\begin{defi}
We say that a curve $\gamma\subset\mathbb R^n$ has \emph{a one-to-one convex projection} (onto the $xy$-plane) if $P_{xy}|_\gamma$ is injective and the projection curve $P_{xy}(\gamma)$ is convex. 
\end{defi}
\subsubsection*{Types of singularity formation}
\begin{defi}
As $t\rightarrow T$, we will say SCSF $\gamma\att$ develops a \emph{Type~I singularity} if
\begin{equation*}
    \limsup_{t\rightarrow T}\sup_{u\in S^1}k^2(u,t)(T-t)<+\infty
\end{equation*}
and a \emph{Type~{II} singularity} otherwise, where $k(u,t)$ is the curvature at the point $\gamma(u,t)$.
\end{defi}
The author proved in \cite{sun2024curve} that if the initial curve $\gamma_0$ has a one-to-one convex projection, then SCSF $\gamma\att$ has a one-to-one convex projection and shrinks to a point as $t\rightarrow T$. A natural question arises whether such SCSF actually develops a Type~I singularity and becomes asymptotically circular as $t\rightarrow T$. We answer this question affirmatively here for a special case, making use of the monotonicity of the symmetric distance ratio (Theorem \ref{n dim-monotonicity of Huisken's distance ratio restricted to symmetric points}).

We restrict to $\mathbb R^3$ for simplicity. 
\begin{thm}
\label{theorem on singularites for symmetric case}
If the initial curve $\gamma_0$ has a one-to-one convex projection onto the $xy$-plane and it is reflection symmetric about the $xz$-plane and $yz$-plane, then SCSF $\gamma\att$ develops a Type~I singularity and becomes asymptotically circular as $t\rightarrow T$.
\end{thm}
\begin{figure}[ht]
\centering
\begin{minipage}{0.3\textwidth}
\centering
\includegraphics[width=\linewidth]{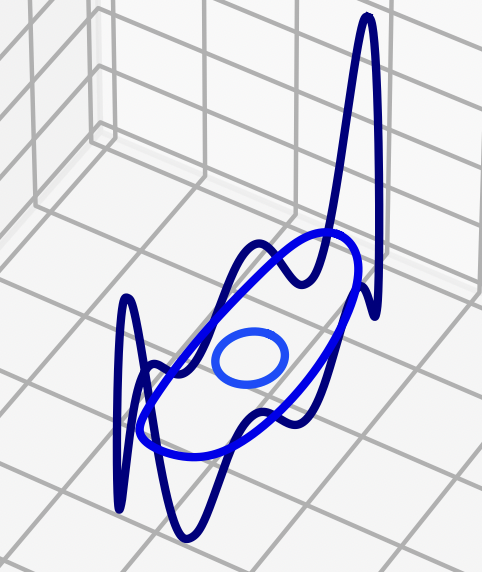}
\end{minipage}\hfill
\begin{minipage}{0.35\textwidth}
\centering
\includegraphics[width=\linewidth]{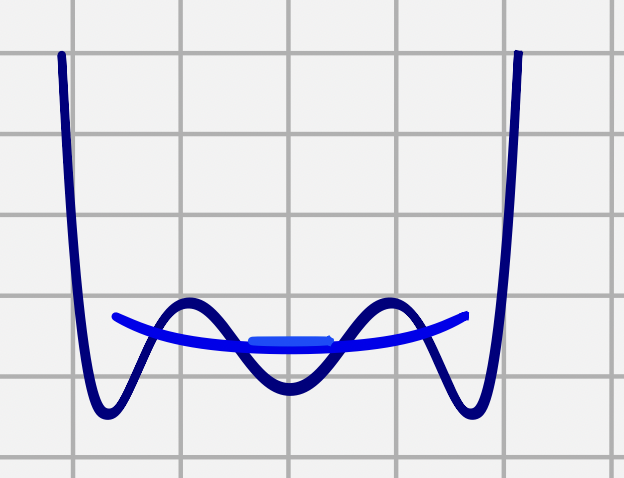}
\textbf{$xz$-projection}
\end{minipage}\hfill
\begin{minipage}{0.25\textwidth}
\centering
\includegraphics[width=\linewidth]{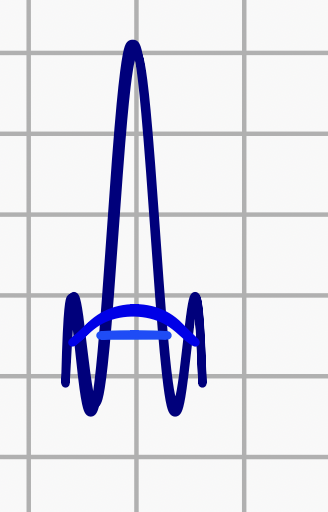}
\textbf{$yz$-projection}
\end{minipage}
\caption{Snapshots of the evolution of 
$(\cos u,0.3\sin u,0.5\cos2u+0.5\cos4u+0.5\cos6u),u\in[0,2\pi]$
from different angles.}
\label{fig:three_images}
\end{figure}
See Figure \ref{fig:three_images} for an example illustrating Theorem \ref{theorem on singularites for symmetric case} numerically.

We will prove Theorem \ref{theorem on singularites for symmetric case} in \S \ref{section on proof of singularities}.

In a forthcoming paper \cite{sun2025singularities}, the author will analyze the singularies of SCSF with convex projections more generally.
\section{Monotonicity of the symmetric distance ratio}
\label{the section on adaptions of Huisken's distance comparison principle}
We first prove the preservation of the symmetric two-crossing condition under SCSF.
\begin{lem}
\label{the lemma on preservation of symmetric two-crossing condition}
If the initial curve $\gamma_0$ is symmetric two-crossing, then the SCSF $\gamma\att$ is also symmetric two-crossing for each time $t\in[0,T)$.
\end{lem}
\begin{proof}
Symmetry is preserved by SCSF, so $\gamma\att$ is also reflection symmetric.

By the Sturmian theorem\cite{angenent1988zero}, the SCSF $\gamma\att$ intersects $P_0$ at no more than two points, see also \cite[Lemma 5.2]{sun2024curve}.

Say $\gamma_0\cap P_0=\{\gamma_0(a),\gamma_0(b)\}$. By symmetry, $\gamma(a,t),\gamma(b,t)\in P_0$ for $t\in[0,T)$.

Thus $\gamma\att$ intersects $P_0$ at exactly two points.
\end{proof}
\subsection*{Notation}
We can track the movement of the intersection points between $\gamma\att$ and $P_0$.
Let us label the intersection points:
\begin{equation}
\label{track the movement of the intersection points}
\gamma\att\cap P_0=\{A(t),B(t)\},  
\end{equation}
where $A,B: [0,T)\rightarrow \mathbb R^n$ are smooth. We may assume that the arc-length parameter\footnote{Because we have chosen the points where $s=0$, we are able to use the arc-length $s=s(q,t)$ as a parameter.} $s$ equals $0$ at the point $A(t)$ for all $t\in[0,T)$.

Recall that we use $L(t)$ to denote the length of the whole curve $\gamma\att$.

Consider two points $\gamma(p,t)$ and $\gamma(q,t)$ which are reflection symmetric about the hyperplane $P_0$ in the sense that
\begin{equation}
\label{notion of symmetric points}
    s_q=-s_p\in[0,\frac{L(t)}{2}],
\end{equation}
where $s_q=s(q,t),s_p=s(p,t)$ are the corresponding arc-length parameters of the points $\gamma(q,t)$ and $\gamma(p,t)$.

For simplicity, let us use $s$ to denote $s_q$. Consequently $s_p=-s$.

With above notation, recall that the intrinsic distance $l=l(p,q)$ is the length of one arc between the points $\gamma(p),\gamma(q)$. We define
\begin{equation}
\label{definition of alpha in terms of s}
\alpha=\frac{\pi l}{L}=\frac{\pi}{L}\int_p^q1ds
=\frac{2\pi}{L}\int_0^q1ds
=\frac{2\pi s}{L}.
\end{equation}
Then $\psi$ (equation (\ref{definition of psi})) can be written as:
\begin{equation}
\label{expressions of psi}
    \psi=\frac{L}{\pi}\sin\alpha.
\end{equation}
The reflection symmetric distance ratio $I$ (Definition \ref{loose definition of the symmetric distance ratio}) can be written in the form:
\begin{equation}
I(q, t) 
= \frac{|\gamma(q,t)-\gamma(p,t)|}{\psi}
=\frac{|\gamma(q,t)-\gamma(p,t)|}
            {\frac{L}{\pi}\sin\left(\frac{2\pi s}{L}\right)}, 
\end{equation}
where
\begin{equation*}
    s=s(q, t)= -s(p,t)\in[0,\frac{1}{2}L(t)].
\end{equation*}
\subsection*{Proof of Theorem \ref{n dim-monotonicity of Huisken's distance ratio restricted to symmetric points}}
Without loss of generality, we may assume 
\[
P_0=\{(x,y,z_1,\cdots,z_{n-2})|y=0\}.
\]
Let $T_0(t)$ be the unit tangent vector at $A(t)$ (equation (\ref{track the movement of the intersection points})). By symmetry, we may assume 
\begin{equation*}
    T_0(t)=(0,1,0,\cdots,0)
\end{equation*}
by appropriately choosing the orientation of the parametrization.

We may also assume $y(q,t)=-y(p,t)\geq0$, thus
\begin{equation}
\label{expressions of I}
    I=I(q,t)=\frac{y(q,t)-y(p,t)}{\psi}
        =\frac{2y(q,t)}{\frac{L}{\pi}\sin\left(\frac{2\pi s}{L}\right)}
\end{equation}
where $s=s_q=s(q,t)\in[0,\frac{1}{2}L(t)]$.

By Hamilton's trick \cite{hamilton1986four} (see also \cite[Lemma 2.1.3]{mantegazza2011lecture}), based on Lemma \ref{the lemma on preservation of symmetric two-crossing condition}, to prove Theorem \ref{n dim-monotonicity of Huisken's distance ratio restricted to symmetric points}, we only need to prove the following proposition.
\begin{prop}
\label{the proposition that I_t is non-negative}
    At any time $t\in[0,T)$, if $q_0$ is a global minimum point of the symmetric distance ratio $I(\cdot,t)$ then $I_t(q_0,t)\geq0$. 
\end{prop}
The rest of the section is devoted to proving Proposition \ref{the proposition that I_t is non-negative}.
\begin{lem}
\label{min of the symmetric distance ratio is no more than 1}
If $q_0$ is a global minimum point of $I(\cdot,t)$ then $I(q_0,t)\leq1$.
\end{lem}
\begin{proof}[Proof of Lemma \ref{min of the symmetric distance ratio is no more than 1}]
Since $T_0(t)=(0,1,0,\cdots,0)$, one can compute:
\begin{equation}
\label{limit of I at s=0}
\lim\limits_{s\rightarrow0}I(s,t)
=\lim\limits_{s\rightarrow0}\frac{2y(s,t)}{\frac{L}{\pi}\frac{2\pi s}{L}}
=\lim\limits_{s\rightarrow0}\frac{y(s,t)}{s}
=\lim\limits_{s\rightarrow0} y_s(s,t)
=y_s(0,t)
=1.
\end{equation}
Thus $$I(q_0,t)\leq\lim\limits_{s\rightarrow0}I(s,t)=1.$$
\end{proof}
Let $q_0$ be a global minimum point of $I(\cdot,t)$. We denote by $s_0:=s_{q_0}$ the corresponding arc-length.

To prove Proposition \ref{the proposition that I_t is non-negative}, we may assume   
\begin{equation}
\label{the equation that s_0 can be assumed between 0 and L over 4}
    s_0\in(0,\frac{1}{4}L(t)]
\end{equation}
for the following reasons.

If $s_{0}=0$, then by the proof of Lemma \ref{min of the symmetric distance ratio is no more than 1}
$$I_t(q_0,t)=\frac{d}{dt}1=0.$$ 
Proposition \ref{the proposition that I_t is non-negative} is proved in this case. Thus we may assume $s_{0}>0$.

The identity $\sin\alpha=\sin(\pi-\alpha)$ implies that the symmetric distance ratio $I$ doesn't depend on the choice of the arcs connecting $\gamma(p,t),\gamma(q,t)$. By swapping $A(t), B(t)$ if needed, we can assume $s_{0}\leq \frac{L}{4}$, where $A(t), B(t)$ are defined in equation (\ref{track the movement of the intersection points}). 
\begin{lem}
If $q_0$ is a global minimum point of $I(\cdot,t)$ then the following holds at $q_0$.
\begin{enumerate}[label=(\alph*)]
 \item By first variation, 
    \begin{equation}
    \label{first variation}
    y_{s}=2\frac{y}{\psi}\cos\alpha.
    \end{equation}
 \item By second variation, 
    \begin{equation}
    \label{second variation}
    y_{ss}\psi+\frac{4\pi}{L} y\sin\alpha\geq0.
    \end{equation}
\end{enumerate}
\end{lem}
\begin{proof}[Proof of the lemma]
By direct computation of derivatives of equation (\ref{expressions of I}),
$$I_{s}=\frac{2}{\psi^2}(y_{s}\psi-\psi_{s}y)$$
and
\begin{align*}
I_{ss}&=\frac{2}{\psi^2}(y_{s}\psi - y\psi_{s})_s + \Bigl(\frac{2}{\psi^2}\Bigr)_s (y_s\psi - y\psi_s) \\
    &=\frac{2}{\psi^2} (y_{ss}\psi-y\psi_{ss}) + \Bigl(\frac{2}{\psi^2}\Bigr)_s (y_s\psi - y\psi_s) .
\end{align*}
By equation (\ref{definition of alpha in terms of s}) and equation (\ref{expressions of psi}), we have
$$
\psi_{s}=2\cos\alpha, \qquad
\psi_{ss} = -\frac{4\pi}{L}\sin\alpha .
$$
Since $q_0$ is a minimum point of $I$ we have $I_s(q_0, t)=0$, which implies
\[
0=y_s\psi-y\psi_s   =y_{s}\psi-2y\cos\alpha. 
\]
Combining with the second variation $I_{ss}(q_0, t)\geq 0$, we have
\[
0\leq y_{ss}\psi - y\psi_{ss} = y_{ss}\psi+\frac{4\pi}{L} y\sin\alpha.
\]
\end{proof}

Recall that we chose the tangential motion of our parametrization $\gamma$ such that $\gamma_t=\gamma_{ss}$ (equation \eqref{equation of CSF}).
\begin{lem}
\label{time derivative of psi}
Under the evolution of SCSF, one has
\begin{equation}
\label{evolution of psi}
    \psi_t\leq-\frac{\sin(\alpha)}{\pi}\frac{4\pi^2}{L} +
    \frac{\alpha\cos\alpha}{\pi}\frac{4\pi^2}{L}-2\cos\alpha\int_0^qk^2ds. 
\end{equation}
\end{lem}
\begin{proof}
Based on
$$L_t=-\int_{S^1}k^2ds$$
and
$$s_t=\left(\int_0^q1ds\right)_t=-\int_0^qk^2ds,$$ 
by direct computation,
\begin{align*}
\psi_t&=-\frac{\sin\alpha}{\pi}\int_{S^1}k^2ds +\frac{L}{\pi}(\cos\alpha) \alpha_t \\
&=-\frac{\sin\alpha}{\pi}\int_{S^1}k^2ds +\frac{L}{\pi}(\cos\alpha)
\left(\frac{2\pi}{L^2}\int_{S^1}k^2ds\int_0^q1ds-\frac{2\pi}{L}\int_0^qk^2ds\right)\\ 
&=-\frac{\sin\alpha}{\pi}\int_{S^1}k^2ds +(\cos\alpha)
\left(\frac{2\pi}{L\pi}\int_{S^1}k^2ds\int_0^q1ds-2\int_0^qk^2ds\right) \\  
&=-\frac{\sin\alpha}{\pi}\int_{S^1}k^2ds +
\frac{\alpha\cos\alpha}{\pi}\int_{S^1}k^2ds-2\cos\alpha\int_0^qk^2ds\\
&=\left(-\frac{\sin\alpha}{\pi} +
\frac{\alpha\cos\alpha}{\pi}\right)\int_{S^1}k^2ds-2\cos\alpha\int_0^qk^2ds.
\end{align*}

Because $\sin\alpha>\alpha\cos\alpha \text{ for }\alpha\in(0,\pi]$, by H\"{o}lder's inequality and Fenchel's theorem, one has 
$$\int_{S^1}k^2ds\geq\frac{1}{L}\left(\int_{S^1}kds\right)^2\geq\frac{4\pi^2}{L}.$$
\end{proof}

\begin{lem}
\label{time derivative of the symmetric distance ratio}
If $q_0$ is a global minimum point of $I(\cdot,t)$ then 
\begin{equation*}
I_t(q_0,t)
\geq\frac{4y\cos\alpha}{\psi^2\int_0^{q_0}1ds}\left((\int_0^{q_0}kds)^2-\alpha^2\right).
\end{equation*}
\end{lem}
\begin{proof}
By computing the time derivative of equation (\ref{expressions of I}) and the second variation (equation (\ref{second variation})),
\begin{equation*}
    I_t(q_0,t)=\frac{2}{\psi^2}(y_t\psi-\psi_t y)=\frac{2}{\psi^2}(y_{ss}\psi-\psi_t y)\geq\frac{2}{\psi^2}(-\frac{4\pi}{L} y\sin\alpha-\psi_t y).
\end{equation*}
Combining Lemma \ref{time derivative of psi}, noticing $y(q_0,t)\geq0$, we have
\begin{equation*}
\begin{split}
    I_t(q_0,t)\geq\frac{2}{\psi^2}\left(\underbrace{-\frac{4\pi}{L} y\sin\alpha+y\frac{\sin(\alpha)}{\pi}\frac{4\pi^2}{L}}_0
    -y\frac{\alpha\cos\alpha}{\pi}\frac{4\pi^2}{L}+2y\cos\alpha\int_0^{q_0}k^2ds\right)\\
    =\frac{4y\cos\alpha}{\psi^2}\left(\int_0^{q_0}k^2ds-\alpha\frac{2\pi}{L}\right)
    =\frac{4y\cos\alpha}{\psi^2\int_0^{q_0}1ds}\left(\int_0^{q_0}1ds\int_0^{q_0}k^2ds-\alpha\frac{2\pi}{L}\int_0^{q_0}1ds\right)\\
    =\frac{4y\cos\alpha}{\psi^2\int_0^{q_0}1ds}\left(\int_0^{q_0}1ds\int_0^{q_0}k^2ds-\alpha^2\right)
    \geq\frac{4y\cos\alpha}{\psi^2\int_0^{q_0}1ds}\left((\int_0^{q_0}kds)^2-\alpha^2\right).
\end{split}
\end{equation*}
\end{proof}

\begin{lem}
\label{geodesic distance is shorter}
If $q_0$ is a global minimum point of $I(\cdot,t)$ then
$$\int_0^{q_0}kds\geq\alpha.$$
\end{lem}
\begin{proof}
The first variation (equation (\ref{first variation})) says that
$$y_{s}=2\frac{y}{\psi}\cos\alpha.$$

Because of Lemma \ref{min of the symmetric distance ratio is no more than 1}, $I(q_0,t)=2\frac{y}{\psi}\leq1$ and $\alpha\in(0,\frac{\pi}{2}]$, thus $$y_{s}=2\frac{y}{\psi}\cos\alpha\leq\cos\alpha,$$
which implies
$$\arccos (y_{s}) \geq\alpha.$$

In addition, the length of the curve $T:[0,q_0]\rightarrow S^{n-1}$ is not less than the geodesic distance between the start point and end point of the curve $T$ on $S^{n-1}$, that is to say $$\int_0^{q_0}|T_s|ds\geq \arccos (T(q_0,t)\cdot T_0),$$
where $T(q_0,t)$ is the unit tangent vector at $\gamma(q_0,t)$ and $T_0=T(0,t)$ is the unit tangent vector at $s=0$.

Noticing that $|T_s|=k$ and $T(q_0,t)\cdot T_0=y_{s}$, $$\int_0^{q_0}kds=\int_0^{q_0}|T_s|ds\geq \arccos (T(q_0,t)\cdot T_0)=\arccos (y_{s}) \geq\alpha.$$
\end{proof}
\begin{proof}[Proof of Proposition \ref{the proposition that I_t is non-negative}]
By Lemma \ref{time derivative of the symmetric distance ratio} and Lemma \ref{geodesic distance is shorter}, at each time $t$, at each global minimum point $q_0$ of $I\att$, $I_t(q_0,t)\geq0$.
\end{proof}

\section{Singularities of Symmetric SCSF with one-to-one convex projections}
\label{section on proof of singularities}

\subsection*{Proof of Theorem \ref{theorem on singularites for symmetric case}}
We assume otherwise, Type~{II} singularities occur as $t\rightarrow T$.
By the blow-up results in \cite{altschuler1991singularities}, up to rescaling, SCSF $\gamma\att$ locally resembles a grim reaper, along some subsequence.

By \cite[Corollary 5.8]{sun2024curve}, the limiting grim reaper cannot be contained in a vertical plane.

Because $\gamma\att$ is reflection symmetric about the $xz$-plane and $yz$-plane, the limiting grim reaper has to be reflection symmetric about one of the $xz$-plane and $yz$-plane, for if it were not true, then $\gamma\att$ contains four disjoint parts that resemble grim reapers and it contradicts that $\gamma\att$ has a one-to-one convex projection onto the $xy$-plane. In more detail, the total curvature of the projection curve is no less than
\begin{equation*}
    4(\pi-\delta)>2\pi\text{ for some small }\delta>0,
\end{equation*}
which is impossible.

Thus the limiting grim reaper is reflection symmetric about the $xz$-plane or $yz$-plane. This contradicts Theorem \ref{n dim-monotonicity of Huisken's distance ratio restricted to symmetric points}.

As a result, SCSF $\gamma\att$ develops a Type~I singularity. Thus by the argument in \cite{huisken1990asymptotic}, any tangent flow is modeled on a self-shrinker. All shrinking curves in $\mathbb R^n$ are planar (see for example \cite[Lemma 5.1]{altschuler2013zoo}).

Among Abresch-Langer curves \cite{abresch1986normalized}, a circle of multiplicity one is the only curve that has a one-to-one convex projection. Then our result follows from \cite{schulze2014uniqueness}.

\subsection*{Acknowledgements}
The author is grateful to his advisors Sigurd Angenent and Hung Vinh Tran for their  guidance, support and their helpful comments on an earlier version.

\bibliographystyle{alpha}
\bibliography{main}
\end{document}